\documentclass[11pt,reqno]{amsart}
\usepackage[T1]{fontenc}
\usepackage{lmodern}
\usepackage{amsmath,amssymb,amsthm,mathtools}
\usepackage[expansion=false]{microtype}
\usepackage[a4paper,margin=30mm]{geometry}
\usepackage{enumitem}
\usepackage[hypertexnames=false,colorlinks=true,linkcolor=blue,citecolor=blue,urlcolor=blue]{hyperref}
\hypersetup{pdftitle={Finite identity bases for flat semirings of linear words},pdfauthor={Aifa Wang}}

\newtheorem{theorem}{Theorem}[section]
\newtheorem{proposition}[theorem]{Proposition}
\newtheorem{lemma}[theorem]{Lemma}
\newtheorem{corollary}[theorem]{Corollary}
\theoremstyle{definition}
\newtheorem{definition}[theorem]{Definition}

\theoremstyle{remark}
\newtheorem{remark}[theorem]{Remark}
\numberwithin{equation}{section}
\newcommand{\V}{\operatorname{Var}}

\newcommand{\supp}{\operatorname{var}}
\newcommand{\NN}{\mathbb N_0}
\newcommand{\eps}{\varepsilon}
\newcommand{\B}{\mathcal B}
\newcommand{\C}{\mathcal C}
\newcommand{\Ainf}{A_{\infty}}
\newcommand{\der}[2]{#1\vdash #2}

\title[Flat semirings of linear words]{Finite identity bases for flat semirings of linear words}
\author{Aifa Wang}
\address{School of Mathematical Sciences, Chongqing University of Technology,\newline Chongqing, 400054, P.R. China}
\email{wangaf@cqut.edu.cn}
\thanks{Corresponding author: Aifa Wang.}
\author{Lili Wang}
\address{School of Mathematical Sciences, Chongqing University of Technology,\newline Chongqing, 400054, P.R. China}
\thanks{This work was supported by the National Natural Science Foundation of China (12371024), the Science and Technology Research Program of Chongqing Municipal Education Commission (KJZD-K2024011102) and the Chongqing Natural Science Foundation Innovation and Development Joint Fund (Municipal Education Commission) (CSTB2025NSCQ-LZX0067).}
\date{}

\begin{document}
\begin{abstract}
For a set $W$ of nonempty words, let $S(W)$ be the flat semiring formed by the nonempty factors of words in $W$, together with an absorbing zero. We prove that $S(W)$ has a finite identity basis whenever every word in $W$ is linear, with no restriction on the size of $W$ or on the lengths of its words. In the bounded case its variety is generated by the interval semiring $A_m\cong S(a_1\cdots a_m)$, where $m$ is the maximum word length and the letters $a_i$ are distinct. In the unbounded case its variety is generated by the interval semiring on all finite intervals of the nonnegative integers. We give explicit finite bases in both cases. The proofs encode nonzero polynomial evaluations by endpoint graphs and derive the required graph identifications using finitely many splicing identities. In particular, the eleven-element semiring $S(abcd)$ is finitely based, providing a counterexample to the length-bound conjectures for $S(W)$ proposed by Gao, Ren and Zhao.
\end{abstract}
\maketitle

\section{Introduction}

An \emph{additively idempotent semiring}, or \emph{ai-semiring}, is an algebra $(S,+,\cdot)$ in which addition is associative, commutative and idempotent, multiplication is associative, and both distributive laws hold. No additive or multiplicative identity is included in this signature. A semiring is \emph{flat} if it has a multiplicative zero $0$ and
\[
 x+x=x,\qquad x+y=0\quad\text{whenever }x\ne y.
\]
Thus $0$ is also the greatest element of the additive semilattice; it is not an additive identity.

The finite basis problem asks whether all identities of an algebra follow from finitely many identities. Flat semirings provide several classes in which the finite basis problem differs substantially from the corresponding problem for the multiplicative reduct. Jackson, Ren and Zhao~\cite{JRZ2022} established nonfinite basis results for flat semirings whose semigroup reducts are finitely based, and proposed the systematic study of the semirings $S(W)$ associated with sets of words. Related questions concerning limit varieties were investigated by Ren, Jackson, Zhao and Lei~\cite{RJZL2023}.

Let $X^+$ be the free semigroup on an alphabet $X$. For a nonempty set $W\subseteq X^+$, write $F(W)$ for the set of all nonempty contiguous factors of words in $W$. The algebra $S(W)$ has underlying set $F(W)\cup\{0\}$, flat addition, and multiplication
\[
u\cdot v=\begin{cases}uv,&uv\in F(W),\\0,&uv\notin F(W).
\end{cases}
\]
Products involving $0$ are zero. We write $S(w)$ for $S(\{w\})$. A word is \emph{linear} if no letter occurs in it more than once. This condition is stronger than avoidance of factors of the form $u^2$.

Gao, Ren and Zhao~\cite{GRZ2026} proved that $S(W)$ generates a Cross variety whenever all words in $W$ have length at most three. They also proved nonfinite basis results when $W$ contains sufficiently large powers. Their Conjectures~6.1 and~6.2 assert that $S(W)$ is finitely based exactly when all words in $W$ have length at most three, for finite nonempty $W$ and arbitrary nonempty $W$, respectively. The linear case has different behavior.

\begin{theorem}\label{thm:main}
Let $W$ be a nonempty set of nonempty linear words. Then $S(W)$ is finitely based in the signature $(+,\cdot)$.
\end{theorem}

The result holds whether or not the lengths of words in $W$ are bounded. To state the more precise conclusion, let
\[
A_m=S(a_1\cdots a_m)\qquad(m\ge1),
\]
where $a_1,\ldots,a_m$ are distinct letters, and let $\Ainf$ be the flat semiring of finite nonempty intervals of $\NN$, defined in Section~\ref{sec:bases}. We prove
\begin{equation}\label{eq:classification}
\V(S(W))=
\begin{cases}
\V(A_m),&m=\max\{|w|:w\in W\}<\infty,\\
\V(\Ainf),&\{|w|:w\in W\}\text{ is unbounded}.
\end{cases}
\end{equation}
Explicit finite bases for these varieties are given below. Consequently, already $W=\{abcd\}$ contradicts both conjectures in~\cite{GRZ2026}.

The proofs are equational. We associate a finite directed graph with each polynomial. A nonzero evaluation in an interval semiring is exactly a strictly increasing assignment of values to the graph vertices. Two splicing identities generate every source-to-sink walk. For $A_m$, one additional identity accounts for the equalities forced by the finite chain of possible endpoint values. For $\Ainf$, no such additional identifications occur. The optional computations described in the appendix are not used in the proofs.

\section{Interval semirings and finite bases}\label{sec:bases}

\subsection{Interval semirings and notation}

For $m\ge1$, identify the factor $a_{i+1}\cdots a_j$ with the symbol $[i,j]$, where $0\le i<j\le m$. Then
\begin{equation}\label{eq:interval}
A_m=\{0\}\cup\{[i,j]:0\le i<j\le m\},\qquad
[i,j][k,l]=\begin{cases}[i,l],&j=k,\\0,&j\ne k.
\end{cases}
\end{equation}
Addition is flat. The same operations define $\Ainf$ on
\[
\{0\}\cup\{[i,j]:i,j\in\NN,\ i<j\}.
\]
These are ai-semirings. Indeed, interval concatenation is associative and is left and right cancellative whenever the products in question are nonzero; these properties give the distributive laws for flat addition. Notice that
\[
|A_m|=1+\binom{m+1}{2}.
\]

We first work in the expanded signature $(+,\cdot,0)$, with absorbing zero. Section~\ref{sec:zero} removes this constant. Modulo the ai-semiring axioms, every term without $0$ is a finite nonempty sum of nonempty words. We identify such a sum with its set $P$ of distinct monomials, and write $\supp(P)$ for the set of its variables. By the zero laws below, every term involving $0$ as a monomial reduces to $0$.

It suffices to derive identities of the form $P\approx P+q$, where $q$ is a word, together with identities $P\approx0$. In fact, $P\approx Q$ is equivalent, over the ai-semiring axioms, to the two finite families
\[
P\approx P+q\ (q\in Q),\qquad Q\approx Q+p\ (p\in P).
\]

\subsection{Finite identity sets}

Let $\C$ consist of the six ai-semiring axioms, the three identities
\begin{equation}\label{eq:Z}\tag{Z}
0+x\approx0,\qquad 0x\approx0,\qquad x0\approx0,
\end{equation}
and the following identities:
\begin{align}
uxv+pxq&\approx uxv+pxq+uxq,\label{eq:M}\tag{M}\\
uv+pv+pq&\approx uv+pv+pq+uq,\label{eq:R}\tag{R}\\
xu+vxz&\approx0,\label{eq:L}\tag{L}\\
ux+vxz&\approx0.\label{eq:E}\tag{E}
\end{align}
The following convention makes these finite sets of identities precise:
\begin{itemize}[leftmargin=*,itemsep=2pt]
\item In \eqref{eq:M}, each of $u,v,p,q$ is either a distinct variable, different from $x$, or is omitted, independently. The common variable $x$ is never omitted. This gives sixteen identities.
\item In \eqref{eq:R}, none of $u,v,p,q$ is omitted.
\item In \eqref{eq:L}, $v$ is not omitted, while $u,z$ may be omitted independently. This gives four identities.
\item In \eqref{eq:E}, $z$ is not omitted, while $u,v$ may be omitted independently. This gives four identities.
\end{itemize}
Omission means deletion from the written word, not evaluation at a unit. Every monomial in the resulting identities is nonempty. In subsequent derivations, a variable standing for a context may be replaced by any nonempty word.

For $m\ge1$, define $\B_m$ by adjoining to $\C$ the identities
\begin{align}
x_1\cdots x_{m+1}&\approx0,\label{eq:N}\tag{$N_m$}\\
x_1\cdots x_m+y_1\cdots y_m
&\approx (x_1+y_1)\cdots(x_m+y_m).\label{eq:H}\tag{$H_m$}
\end{align}
Define $\B_\infty$ by adjoining to $\C$, instead, the identities
\begin{equation}\label{eq:J}\tag{J}
x^2\approx0,\qquad xyx\approx0.
\end{equation}
In particular, each $\B_m$ and $\B_\infty$ is a finite set.

\begin{theorem}\label{thm:bases}
In the signature $(+,\cdot,0)$, $\B_m$ is an identity basis for $A_m$ for every $m\ge1$, and $\B_\infty$ is an identity basis for $\Ainf$.
\end{theorem}

We prove validity first and completeness in Sections~\ref{sec:finite} and~\ref{sec:infinite}.

\begin{lemma}\label{lem:sound}
Every $A_m$ and $\Ainf$ satisfies $\C$. Moreover, $A_m\models\B_m$ and $\Ainf\models\B_\infty$.
\end{lemma}
\begin{proof}
The ai-semiring axioms and \eqref{eq:Z} hold by construction. For an absorption identity, if the left-hand sum is zero, the right-hand sum is zero as well. Thus it is enough to consider assignments for which all summands on the left are equal and nonzero.

For \eqref{eq:M}, suppose that $uxv$ and $pxq$ evaluate to the same interval $I$, and that $x$ evaluates to $[i,j]$. Both occurrences of $x$ have these same endpoints. The left contexts, when present, are consequently the same initial part of $I$, and the right contexts are the same final part. If a context is omitted, the corresponding endpoint of $[i,j]$ is an endpoint of $I$. In all cases the mixed product $uxq$ also evaluates to $I$.

For \eqref{eq:R}, equality $uv=pv\ne0$ implies that the values of $u$ and $p$ coincide, by right cancellation of nonzero interval products. Hence $uq=pq$.

For \eqref{eq:L}, the interval assigned to $x$ starts at the left endpoint of the value of $xu$. In $vxz$, its left endpoint is strictly to the right of the left endpoint of the whole product, since $v$ is nonempty. The two products cannot be equal and nonzero. The argument for \eqref{eq:E} uses right endpoints.

In $A_m$, every nonzero factor has positive interval length, so a product of $m+1$ factors is zero. If two products of $m$ factors are equal and nonzero, their common value must be $[0,m]$, and their $i$th factors must both equal $[i-1,i]$. Thus every monomial in the expansion of the right side of \eqref{eq:H} has the same value. If the left side is zero, the expansion on the right contains both original monomials and is zero. This proves \eqref{eq:N} and \eqref{eq:H}.

Finally, in $\Ainf$ an interval cannot occur twice in a nonzero concatenation, because endpoint values strictly increase along the concatenation. This proves \eqref{eq:J}.
\end{proof}

\section{Endpoint graphs and splicing}\label{sec:splicing}

\subsection{The endpoint graph}

\begin{definition}\label{def:graph}
Let $P$ be a finite nonempty set of nonempty words. Introduce symbols $s,t$ and, for each $x\in\supp(P)$, symbols $\lambda_x,\rho_x$. For each word $x_1\cdots x_r\in P$, impose the identifications
\begin{equation}\label{eq:endpoints}
\lambda_{x_1}=s,\qquad \rho_{x_i}=\lambda_{x_{i+1}}\ (1\le i<r),\qquad \rho_{x_r}=t.
\end{equation}
The vertices of $G(P)$ are the classes of the equivalence relation generated by these identifications. For each variable $x$, there is one directed edge labelled $x$ from $\lambda_x$ to $\rho_x$. We use the endpoint symbols also for their classes.
\end{definition}

The graph may have loops and parallel edges. A \emph{walk} may repeat edges and vertices; a \emph{path} does not repeat vertices. Each monomial of $P$ labels a nonempty $s$--$t$ walk, and every vertex and edge lies on such a walk.

\begin{lemma}\label{lem:evaluation}
Nonzero evaluations of $P$ in $A_m$, restricted to $\supp(P)$, are in bijection with maps
\[
f:V(G(P))\longrightarrow\{0,\ldots,m\}
\quad\text{such that}\quad f(\lambda_x)<f(\rho_x)\quad(x\in\supp(P)).
\]
The evaluation associated with $f$ sends $x$ to $[f(\lambda_x),f(\rho_x)]$ and sends $P$ to $[f(s),f(t)]$. The corresponding statement for $\Ainf$ uses maps into $\NN$.
\end{lemma}
\begin{proof}
A nonzero sum in a flat semiring has all its summands equal and nonzero. The multiplication rule~\eqref{eq:interval} says exactly that the endpoints in~\eqref{eq:endpoints} agree under such an evaluation. Every occurring variable is assigned a nonzero interval, so its left endpoint is strictly smaller than its right endpoint. Conversely, these conditions make every monomial evaluate to $[f(s),f(t)]$. The endpoints determine the assigned intervals uniquely.
\end{proof}

\begin{corollary}\label{cor:zero-criterion}
The polynomial $P$ is identically zero in $A_m$ if and only if $G(P)$ has a directed cycle or has an $s$--$t$ path of length greater than $m$. It is identically zero in $\Ainf$ if and only if $G(P)$ has a directed cycle.
\end{corollary}
\begin{proof}
A directed cycle precludes a strictly increasing map. A path of length greater than $m$ precludes such a map into $\{0,\ldots,m\}$. Conversely, in a finite acyclic graph of the present form, assign to each vertex the maximum length of a path from $s$ to it. This is strictly increasing along each edge and its largest value is the maximum length of an $s$--$t$ path.
\end{proof}

\subsection{Deriving the walks of the graph}

Throughout the following argument, a derivation may use the ai-semiring axioms and \eqref{eq:Z}, as well as the indicated splicing identities. Adding a derived monomial to $P$ means deriving an identity $P\approx P+q$. Previously added monomials can be used in later steps, by associativity and idempotence of addition.

\begin{lemma}[Splicing lemma]\label{lem:splicing}
For every polynomial $P$, either $\der{\C}{P\approx0}$, or, for every nonempty $s$--$t$ walk with label $q$ in $G(P)$,
\[
\der{\C}{P\approx P+q}.
\]
\end{lemma}
\begin{proof}
We give the derivation in three steps.

\smallskip\noindent\emph{Step 1: boundary conflicts.}
Suppose a variable $x$ occurs first in one monomial and not first in another. The two monomials have the forms $xu$ and $vxz$, where $v$ is nonempty and $u,z$ may be empty. An instance of \eqref{eq:L} makes their sum zero, hence $\C\vdash P\approx0$. A variable occurring both last and not last is dealt with by \eqref{eq:E}. The two occurrences are allowed to belong to the same monomial, using additive idempotence.

We may therefore assume that neither conflict occurs. The class of $s$ consists only of $s$ and the left endpoints of variables occurring first. Indeed, to leave these symbols along a chain of the generating identifications, one would need an adjacency identification involving the left endpoint of a variable that also occurs first. That would give a nonfirst occurrence of this variable. Similarly, the class of $t$ consists only of $t$ and the right endpoints of variables occurring last. In particular, $s\ne t$, no edge enters $s$, and no edge leaves $t$.

\smallskip\noindent\emph{Step 2: splicing at an internal vertex.}
An internal cut of a monomial is a pair $(A,B)$ of nonempty words with $AB\in P$, together with its occurrence in that monomial. Two internal cuts represent the same vertex of $G(P)$ precisely when they can be joined by a chain of cuts in which consecutive cuts have either the same immediately preceding variable or the same immediately following variable. This follows from the definition of the endpoint identifications and the exclusion of boundary conflicts.

For two directly related cuts $(A,B)$ and $(C,D)$, \eqref{eq:M} adds both $AD$ and $CB$. If the common preceding variable is $x$, write the two words as $A'xB$ and $C'xD$; if the common following variable is $x$, write them as $AxB'$ and $CxD'$. The permitted omissions in \eqref{eq:M} cover empty $A',C',B',D'$.

Transitivity is supplied by \eqref{eq:R}. More explicitly, if cuts indexed by $i,j,k$ have prefixes $A_i,A_j,A_k$ and suffixes $B_i,B_j,B_k$, and the words $A_iB_j$ and $A_jB_k$ have already been added, then the original word $A_jB_j$ and the instance
\[
A_iB_j+A_jB_j+A_jB_k
\approx A_iB_j+A_jB_j+A_jB_k+A_iB_k
\]
add $A_iB_k$. All four substituted contexts are nonempty. Induction along a finite chain of cuts shows that for any two internal cuts representing the same vertex, the prefix of the first can be joined to the suffix of the second.

Every word added in this step is a walk in the original graph. Its endpoint identifications are therefore redundant, so adjoining it does not change $G(P)$. The argument can consequently be used again after any finite number of such additions.

\smallskip\noindent\emph{Step 3: an arbitrary walk.}
Let $q=x_1\cdots x_r$ label an $s$--$t$ walk. By Step~1, some original monomial starts with $x_1$. Choose such a monomial. Suppose a word beginning with $x_1\cdots x_i$ has been obtained, with $i<r$. Its suffix after this prefix is nonempty: otherwise $\rho_{x_i}=t$, although the target walk continues from that vertex, contradicting the absence of edges leaving $t$.

Choose an occurrence of $x_{i+1}$ in an original monomial. It is not first, since its left endpoint is the current internal vertex. Its preceding cut and the cut after the constructed prefix represent the same vertex. Step~2 joins the constructed prefix to the suffix beginning with this occurrence of $x_{i+1}$. Continue inductively. Finally $\rho_{x_r}=t$, so every occurrence of $x_r$ in the original polynomial is last, again by Step~1. The final constructed word is exactly $q$.
\end{proof}

\begin{corollary}\label{cor:zero-derived}
If $A_m\models P\approx0$, then $\B_m\vdash P\approx0$.
\end{corollary}
\begin{proof}
If Step~1 of Lemma~\ref{lem:splicing} applies, the conclusion follows there. Otherwise, by Corollary~\ref{cor:zero-criterion}, the graph has an $s$--$t$ walk of length at least $m+1$. In the cyclic case, such walks are obtained by traversing a cycle sufficiently many times; every vertex is reachable from $s$ and can reach $t$. The splicing lemma adds the corresponding word. Identity~\eqref{eq:N} makes that word zero, and \eqref{eq:Z} gives $P\approx0$.
\end{proof}

\section{Completeness for finite interval semirings}\label{sec:finite}

\subsection{Equalities forced by a finite chain}

Let $G$ be a finite acyclic directed graph with distinguished vertices $s,t$ such that every vertex lies on an $s$--$t$ path. Suppose its maximum $s$--$t$ path length is at most $m$. Define
\[
r(v)=\max\{|p|:p\text{ is an }s\text{--}v\text{ path}\},\qquad
d(v)=\max\{|p|:p\text{ is a }v\text{--}t\text{ path}\}.
\]
Call $v$ \emph{critical} if $r(v)+d(v)=m$. Equivalently, $v$ lies on an $s$--$t$ path of length $m$. A map from $V(G)$ to $\{0,\ldots,m\}$ is \emph{admissible} if it is strictly increasing along every edge.

\begin{lemma}\label{lem:forced}
For $u,v\in V(G)$, every admissible map $f$ satisfies $f(u)=f(v)$ if and only if either $u=v$, or $u,v$ are critical and $r(u)=r(v)$.
\end{lemma}
\begin{proof}
An admissible map on a path of length $m$ must use, successively, all the values $0,\ldots,m$. Thus every critical vertex $v$ has $f(v)=r(v)$, proving one implication.

For the converse, $f(z)=r(z)$ is an admissible map. It separates $u$ and $v$ if their ranks differ. Suppose $u\ne v$ and $r(u)=r(v)=k$. Then neither vertex is reachable from the other. If one of them, say $u$, is not critical, then $k+d(u)\le m-1$.

For a vertex $z$ reachable from $u$, let $\ell(u,z)$ be the maximum length of a path from $u$ to $z$, including the length-zero path when $z=u$. Define
\begin{equation}\label{eq:separation}
f(z)=
\begin{cases}
\max\{r(z),k+1+\ell(u,z)\},&u\leadsto z,\\
r(z),&u\not\leadsto z.
\end{cases}
\end{equation}
This map takes values at most $m$, since $\ell(u,z)\le d(u)$. It is strictly increasing along every edge. If the initial vertex of an edge is reachable from $u$, both lower bounds in the maximum increase by at least one along that edge. If it is not reachable, the strict increase of $r$ suffices. Finally, $f(u)=k+1$ whereas $f(v)=k$. Thus $u,v$ cannot have equal values under every admissible map unless they are both critical.
\end{proof}

Write $u\sim_m v$ for the equivalence relation described in Lemma~\ref{lem:forced}.

\begin{lemma}\label{lem:mixing}
If $P$ is not identically zero in $A_m$, then there is a polynomial $P'$ such that
\[
\B_m\vdash P\approx P',\qquad G(P')\cong G(P)/{\sim_m},
\]
where the isomorphism preserves the labelled edges and the distinguished vertices.
\end{lemma}
\begin{proof}
The graph $G(P)$ is acyclic and has maximum path length at most $m$. By Lemma~\ref{lem:splicing}, add to $P$ all monomials labelling its $s$--$t$ paths of length $m$. There are finitely many, and their addition does not change the graph. If there are none, take $P'=P$.

Otherwise apply \eqref{eq:H} to each pair of these monomials and expand the product of sums. Let $P'$ be the resulting polynomial, including all previous monomials. These are finite equational derivations.

Every edge in a length-$m$ path joins critical vertices of successive ranks. Each mixed word therefore imposes only equalities between critical vertices of the same rank. Conversely, let $u,v$ be critical vertices of rank $k$, with $0<k<m$. Choose length-$m$ paths through them. The mixed word consisting of the first $k$ edges of the first path followed by the final $m-k$ edges of the second path is among the added words. Its middle adjacency identifies $u$ and $v$. At ranks $0$ and $m$, the only vertices are $s$ and $t$. Hence the newly imposed equivalence relation is exactly $\sim_m$.
\end{proof}

\subsection{Derivation of all identities}

\begin{proof}[Proof of Theorem~\ref{thm:bases} for $A_m$]
Validity is Lemma~\ref{lem:sound}. Consider an identity $P\approx P+q$ valid in $A_m$, where $q$ is a word. If $P$ is identically zero, Corollary~\ref{cor:zero-derived} proves the identity from $\B_m$.

Suppose $P$ is not identically zero. Then $\supp(q)\subseteq\supp(P)$. Otherwise keep a nonzero evaluation of $P$ and assign zero to a variable occurring in $q$ but not in $P$; this contradicts the identity.

Write $q=x_1\cdots x_r$. Under every nonzero evaluation of $P$, the word $q$ must have the same nonzero value as $P$. By Lemma~\ref{lem:evaluation}, this is equivalent to requiring, under every admissible endpoint map,
\begin{equation}\label{eq:qconditions}
f(\lambda_{x_1})=f(s),\qquad
f(\rho_{x_i})=f(\lambda_{x_{i+1}})\ (1\le i<r),\qquad
f(\rho_{x_r})=f(t).
\end{equation}
Lemma~\ref{lem:forced} identifies each required equality with a pair in $\sim_m$. By Lemma~\ref{lem:mixing}, all these pairs have become actual vertex equalities in $G(P')$, for a polynomial $P'$ satisfying $\B_m\vdash P\approx P'$. Thus $q$ labels an $s$--$t$ walk of $G(P')$.

The splicing lemma applied to $P'$ yields $\C\vdash P'\approx P'+q$, unless it yields $P'\approx0$. The latter is impossible, since all the preceding identities are valid in $A_m$ and $P$ has a nonzero evaluation. Consequently $\B_m\vdash P\approx P+q$.

Identities with one side zero are covered by Corollary~\ref{cor:zero-derived}. Decomposing any remaining identity into absorption identities completes the proof.
\end{proof}

\section{The unbounded interval semiring}\label{sec:infinite}

\begin{proof}[Proof of Theorem~\ref{thm:bases} for $\Ainf$]
Again validity follows from Lemma~\ref{lem:sound}. Suppose first that $\Ainf\models P\approx0$. By Corollary~\ref{cor:zero-criterion}, $G(P)$ has a directed cycle. A boundary conflict gives $\C\vdash P\approx0$ immediately. Otherwise choose an edge labelled $x$ on a cycle. There is an $s$--$t$ walk traversing this edge at least twice, since every vertex lies on an $s$--$t$ walk. Its label has a factor $xUx$, where $U$ may be empty. The splicing lemma adds this word to $P$. If $U=\eps$, use $x^2\approx0$; if $U\ne\eps$, substitute $U$ for $y$ in $xyx\approx0$. In either case the entire added monomial is zero, and $\B_\infty\vdash P\approx0$.

Now suppose $P\approx P+q$ is valid and $P$ is not identically zero. As before, every variable of $q$ occurs in $P$. The graph $G(P)$ is finite and acyclic, so it has an injective topological numbering by nonnegative integers. This numbering is an admissible endpoint map for $\Ainf$ and separates every pair of distinct vertices. Therefore the equalities~\eqref{eq:qconditions}, which must hold under this map, can only relate vertices that are already equal in $G(P)$. It follows that $q$ labels an $s$--$t$ walk in $G(P)$. Lemma~\ref{lem:splicing} derives $P\approx P+q$ from $\C$. The alternative $P\approx0$ is excluded by validity and the chosen nonzero evaluation.

These two cases derive all identities of $\Ainf$ from $\B_\infty$.
\end{proof}

\begin{remark}
The identities $x^2\approx0$ and $xyx\approx0$ do not bound the length of nonzero products of distinct elements. In particular, $\Ainf$ is not nilpotent. The role of \eqref{eq:J} is to eliminate repeated edges in a cycle, not to impose a uniform bound on path length.
\end{remark}

\section{Removing the constant and treating word sets}\label{sec:zero}

\subsection{The original binary signature}

\begin{proposition}\label{prop:constant}
Each of the algebras $A_m$ and $\Ainf$ has a finite identity basis in the signature $(+,\cdot)$.
\end{proposition}
\begin{proof}
For finite $m$, replace every occurrence of the constant $0$ in $\B_m$ by $z^{m+1}$, using a fresh variable $z$ in each identity containing the constant (first renaming an existing variable called $z$, if necessary). Keep all identities without $0$. In particular, the translation of \eqref{eq:N} is
\[
x_1\cdots x_{m+1}\approx z^{m+1}.
\]
In any nonempty algebra satisfying these identities, the value of $z^{m+1}$ is independent of $z$: substitute the same variable for all $x_i$ and vary the two assignments. Denote the common value by $0$. The translations of \eqref{eq:Z} make it an absorbing element for both operations. Every translated identity is then precisely the corresponding identity in $\B_m$ under this interpretation. Thus the algebra has an expansion satisfying $\B_m$.

The translated identities hold in the binary reduct of $A_m$. Conversely, by Theorem~\ref{thm:bases}, every algebra satisfying the translated identities has an expansion in $\V(A_m)$ in the expanded signature. Its binary reduct therefore satisfies every binary identity of $A_m$. This proves the assertion for $A_m$.

For $\Ainf$, replace $0$ by a fresh square $z^2$. The translation of $x^2\approx0$ is $x^2\approx z^2$, which supplies the common value. The same argument applies to $\B_\infty$.
\end{proof}

\begin{lemma}\label{lem:join}
For every nonempty set $W$ of nonempty words,
\[
\V(S(W))=\bigvee_{w\in W}\V(S(w)).
\]
\end{lemma}
\begin{proof}
For $w\in W$, the map $\pi_w:S(W)\to S(w)$ retains each factor of $w$ and sends every other element to zero. It is a surjective homomorphism. For multiplication, if a concatenation is a factor of $w$, then each of its two factors is a factor of $w$. For addition, distinct retained elements remain distinct, and all other cases follow from zero absorption.

The family $(\pi_w)_{w\in W}$ separates distinct elements. Given a nonzero factor $u$, choose $w$ containing it. Its image under $\pi_w$ is nonzero and differs from the image of any other element, which is either a different factor or zero. Thus $S(W)$ is a subdirect product of the $S(w)$, and the asserted equality follows. This is also the decomposition used in~\cite[Proposition~3.12]{GRZ2026}.
\end{proof}

\begin{proof}[Proof of Theorem~\ref{thm:main} and \eqref{eq:classification}]
If $w$ is linear, renaming its letters gives $S(w)\cong A_{|w|}$. For $k\le m$, $A_k$ is both a subsemiring of $A_m$ and a homomorphic image of $A_m$: the latter map retains intervals contained in $[0,k]$ and sends the others to zero.

If the lengths in $W$ are bounded and their maximum is $m$, Lemma~\ref{lem:join} therefore gives $\V(S(W))=\V(A_m)$.

If the lengths are unbounded, the same lemma gives
\[
\V(S(W))=\bigvee_{m\ge1}\V(A_m).
\]
Every $A_m$ is a subsemiring of $\Ainf$, and every finite set of elements of $\Ainf$ is contained in some $A_m$. Hence an identity holds in $\Ainf$ exactly when it holds in every $A_m$. The join above is consequently $\V(\Ainf)$. Proposition~\ref{prop:constant} now proves finite basability in both cases.
\end{proof}

\begin{corollary}\label{cor:counterexample}
The semiring $S(abcd)$, where $a,b,c,d$ are distinct, is an eleven-element noncommutative finitely based flat semiring. Hence Conjectures~6.1 and~6.2 of~\cite{GRZ2026} are false.
\end{corollary}
\begin{proof}
Apply Theorem~\ref{thm:main} to $W=\{abcd\}$. Its ten nonzero factors are
\[
a,b,c,d,\ ab,bc,cd,\ abc,bcd,\ abcd.
\]
Moreover, $a\cdot b=ab\ne0=b\cdot a$. The word has length four, so it contradicts each conjectured necessary condition.
\end{proof}

\begin{corollary}
The varieties $\V(A_m)$ form a strictly increasing chain of finitely based varieties, and their join $\V(\Ainf)$ is also finitely based.
\end{corollary}
\begin{proof}
The inclusions follow from the interval embeddings. The binary identity
\[
x_1\cdots x_{m+1}\approx z^{m+1}
\]
holds in $A_m$ but fails in $A_{m+1}$: assign $x_i\mapsto[i-1,i]$ and $z\mapsto0$. The assertions about bases and the join have already been proved.
\end{proof}

\section{Scope of the linear-word hypothesis}

The following simple identity identifies the class of word semirings to which Theorem~\ref{thm:main} applies.

\begin{proposition}\label{prop:linear-characterization}
For every nonempty set $W$ of nonempty words,
\[
S(W)\models xy+zx\approx0
\quad\Longleftrightarrow\quad
\text{every word in }W\text{ is linear}.
\]
Here the displayed identity is read in the natural expansion by the absorbing zero of $S(W)$.
\end{proposition}
\begin{proof}
Suppose all words in $W$ are linear. If $xy+zx$ had a nonzero value, the two products would be equal to a common factor $v$. The nonempty word assigned to $x$ would be both a prefix and a suffix of $v$, and $|v|>|x|$. The first letter of this prefix would therefore occur in at least two positions of $v$, contrary to linearity.

Conversely, if a word in $W$ has a repeated letter $a$, it has a factor $aUa$ for some possibly empty word $U$. Assign $x\mapsto a$, $y\mapsto Ua$ and $z\mapsto aU$. All three assigned words are nonempty factors, including when $U=\eps$. Then $xy=zx=aUa\ne0$, so the identity fails.
\end{proof}

The identity in Proposition~\ref{prop:linear-characterization} is an instance of a boundary identity in $\C$. Consequently the full bases given here cannot be transferred unchanged to a word containing a repeated letter. The results do not classify the finite basis property of the remaining length-four patterns. They establish a family of finitely based word semirings of arbitrarily large nilpotency class, and show that word length alone cannot supply the proposed dichotomy.

\appendix
\section{Reproducible checks}

The proofs above do not rely on computation. The JavaScript supplement checks \eqref{eq:M}, \eqref{eq:R}, \eqref{eq:L} and \eqref{eq:E} on $A_m$ for $1\le m\le5$: 125 identity instances and 4,476,680 assignments. It compares the endpoint criterion with direct evaluations on 2,000 generated instances and 13,950 exhaustive two-letter instances for $A_4$. Syntactic saturation is checked on 3,429 further polynomials, as well as longer cut-chain and directed-chain examples.

A separate syntactic checker, using neither endpoint graphs nor semiring evaluations, verifies 1,214 finite-height certificates with 2,467 steps and eight unbounded certificates with 115 steps; it rejects four deliberately invalid certificates. Lemma~\ref{lem:forced} is additionally tested on 135 directed acyclic graphs with two to five vertices, by enumerating 3,913 admissible maps with bounds at most five. The supplementary README specifies the enumeration ranges, commands and certificate format. These finite checks can detect local errors but are neither a proof-assistant formalization nor substitutes for the general arguments.

\end{document}